\documentclass[]{imsart}

\RequirePackage{amsthm,amsmath,amsfonts,amssymb}
\RequirePackage{natbib}
\newcommand{\bTheta}{\boldsymbol{\Theta}}
\newcommand{\norm}[1]{\lVert#1\rVert}

\usepackage{graphicx}
\usepackage{amssymb}
\usepackage{epstopdf}
\usepackage{multirow} 
\usepackage{tikz}
\usepackage{subfigure,mathrsfs}
\usepackage{threeparttable}
\newtheorem{theorem}{Theorem}[section]

\newtheorem{lemma}[theorem]{Lemma}
\usepackage{amsmath,amssymb,amstext,amsthm} 

\newcommand{\el}{e\ell}
\newcommand{\bbE}{\mathbb E}

\startlocaldefs

\endlocaldefs

\begin{document}

\begin{frontmatter}
\title{Global Consistency of Empirical Likelihood}
\runtitle{Global consistency}

\begin{aug}
\author[A]{\fnms{Haodi}~\snm{Liang}},
\and
\author[A]{\fnms{Jiahua}~\snm{Chen}}
\address[A]{Department of Statistics, University of British Columbia\printead[presep={\ }]{}}

\end{aug}

\begin{abstract}
This paper develops several interesting, significant, and interconnected approaches
to non-parametric or semi-parametric statistical inferences.
The overwhelmingly favoured maximum likelihood estimator (MLE) under
parametric model is 
renowned for its strong consistency and optimality generally credited to \cite{Cramer1946}.
These properties, however, falter when the model is not regular or not completely accurate.
In addition, their applicability is limited to local maxima close to the unknown true parameter value.
One must therefore ascertain that the global maximum of the likelihood is strongly consistent under generic conditions \citep{Wald1949}.
Global consistency is also a vital research problem in the context of empirical likelihood \citep[EL]{Owen2001}. 
The EL is a ground-breaking platform for non-parametric statistical inference.
A subsequent milestone is achieved by placing estimating functions under the EL umbrella \citep{QinLawless1994}. 
The resulting profile EL function possesses many nice properties of parametric likelihood but also shares the same shortcomings.
These properties cannot be utilized unless we know the local maximum at hand is close to the unknown true parameter value. 
To overcome this obstacle, we first put forward a clean set of conditions under which the global maximum is consistent. 
We then develop a {\sc global maximum test} to ascertain if the local maximum at hand is in fact a global maximum.
Furthermore, we invent a {\sc global maximum remedy} to ensure global consistency by expanding the set of estimating functions under EL.
Our simulation experiments on many examples from the literature firmly establish that the proposed approaches work as predicted.
Our approaches also provide superior solutions to problems of their parametric counterparts investigated 
by \cite{DeHaan}, \cite{Veall}, and \cite{Jiang1999}.

\end{abstract}

\end{frontmatter}


\section{Introduction}
Suppose we have a set of independent and identically distributed (iid) observations 
$x_1, \ldots, x_n$ from a distribution $F$ whose density function with respect to some
$\sigma$-finite measure is a member of $\{f(x; \theta): \theta \in \Theta\}$ for some 
parametric space $\Theta \in {\mathbb R}^q$ and positive integer $q$.
The log likelihood function of the parameter is then given by
\begin{equation}
\label{para.lik.f}
\ell_n (\theta) = \sum_{i=1}^n \log f(x_i; \theta).
\end{equation}
One example of $\{f(x; \theta): \theta \in \Theta\}$ is the Cauchy distribution family with
$q=1$ so that $\log f(x; \theta) = - \log \{ 1 + (x-\theta)^2\}$. 
One may estimate $\theta$ by its maximum likelihood estimator (MLE) $\hat \theta_n$ 
defined by
\[
\ell_n (\hat \theta_n ) = \sup_{\theta \in \Theta}  \ell_n (\theta).
\]
When $f(x; \theta)$ is smooth in $\theta$, define the score function
\[
s(x; \theta) = \frac{\partial \log f(x; \theta) }{\partial \theta}.
\]
One may instead estimate $\theta$ by $\tilde \theta_n$ that solves the score
equation
\begin{equation}
\label{score.f}
\sum_{i=1}^n s(x_i;  \theta) = 0.
\end{equation}
Under many commonly used models (but not Cauchy), $\hat \theta_n = \tilde \theta_n$.
It is generally stated that the MLE $\hat \theta_n$ is strongly consistent for 
$\theta$ \citep{Wald1949}
and $\sqrt{n} (\tilde \theta_n - \theta)$ is asymptotically normal with 0-mean
and the lowest variance \citep{Cramer1946}. 

Under the Cauchy model, however, the score equation \eqref{score.f}
has multiple solutions \citep{Small}. 
We do not generally have $\hat \theta_n = \tilde \theta_n$.
The consistency conclusion of \cite{Wald1949} is applicable to global maximum $\hat \theta_n$,
not to all $\tilde \theta_n$ because many of them are local maxima.
The asymptotic normality of \cite{Cramer1946} only applies to these $\tilde \theta_n$
that are within $o_p(1)$ distance of the true value of $\theta$,
and $\hat \theta_n$ is emphatically included.
In the Cauchy example, one can easily identify one $\tilde \theta_n$ that equals
$\hat \theta_n$. 
Yet the general problem requires careful investigation \citep{Jiang1999}.

Suppose we do not wish to postulate a parametric model for $F$.
Let $p_i = P(X = x_i)$ with the probability calculated under $F$
for $i=1, 2, \ldots, n$.
The log-empirical likelihood function  of $F$  \citep[EL]{Owen2001} is then given by
\[
\el_n(F)=\sum_{i=1}^{n} \log p_i.
\]
At the same time, consider the
situation where the parameter of interest $\theta$ is a functional of $F$
defined by
\[
{\mathbb E}_F \{ g(X;  \theta)\}= {\mathbb E} \{ g(X;  \theta): F\} = 0
\]
for some vector valued estimating function of both $X$ and $\theta$, 
where ${\mathbb E}$ is computed when $X$ has distribution $F$.
We may omit  $F$ in the notation subsequently for notational simplicity.
\cite{QinLawless1994} propose to define a profile log EL function
\begin{equation}
\label{EL.f}
\el_n(\theta) 
= \sup \{  \sum_{i=1}^{n} \log p_i:~ \sum_{i=1}^n p_i = 1,
\sum_{i=1}^n p_i g(x_i; \theta) = 0\}.
\end{equation}
We may then estimate $\theta$ by the maximum EL estimator (MELE)
defined to be 
\[
\hat \theta_n= \arg\max \el_n(\theta)
\]
 and
construct its asymptotic confidence regions via chi-square approximation
to a likelihood ratio statistic.
Like its parametric counterpart, $\hat \theta_n$ is known to be  consistent only
for the one within $n^{-1/3}$-neighbourhood of the $\theta$ \citep{QinLawless1994}.
If we choose the estimating function to be the score function of the Cauchy model,
there are multiple global maxima and only one of them is consistent.
We will deliberate this fact later. 
The existence of such examples leads to a vital question: 
when is a global maximum consistent?
If the global maximum is not consistent, does there exist an effective remedy? 
Once this problem is solved, a new task is
judging whether or not a local maximum of $\el_n(\theta)$ is in fact 
a global maximum, much like \cite{DeHaan}, \cite{Veall} and  \cite{Jiang1999}.

Our research has produced satisfactory answers to all the questions raised above.
We first sort out conditions on estimating functions under which the 
global maximum of $\el_n(\theta)$ (the MELE) is strongly consistent. 
This result significantly extends the classical results on M-estimators by \cite{Huber}.
If $\el_n(\theta)$ does not produce a consistent global maximum with
the current set of estimating functions,
 we propose a generic approach to achieving global consistency
through expanding the set of estimating functions.
In addition, we come up with a method to judge if a local maximum is in fact a global maximum.
Our method takes advantage of the special properties of the EL. 
Unlike \cite{Jiang1999}, it is free from the burden of estimating any scaling parameters.
Furthermore, it provides a much more powerful alternative solution to problems
discussed in \cite{Veall},  \cite{DeHaan}, and \cite{Jiang1999}.

\section{Global consistency of the MELE}

We first give a more detailed but still brief introduction of the profile empirical likelihood. 
Unless otherwise specified, we study the inference problems given
a set of independent and identically distributed (iid) observations
$x_1, \ldots, x_n$ from a distribution $F$. 
When it is essential to distinguish between the random variable and its observed values,
we use $X$ and $X_i$ for the generic and specific random variables.
Otherwise, we may regard $x$ and $x_i$ both as random variables and
their observed values. 
Let the parameter of interest $\theta=\theta(F)$ be the unique solution 
to $\bbE \{ g(X; \theta)\}=0$ for some $m$-dimensional bivariate function $g(x;\theta)$
and $q$ dimensional $\theta$.
The expectation here is calculated under the promise that the distribution
of $X$ is given by $F$. 
If $\theta$ is the population mean, we may define it by choosing $g(x; \theta)=x-\theta$. 
If $F$ is known to have equal mean and variance, we may choose
\[
g(x; \theta) 
=\begin{pmatrix}
x-\theta\\
x^2-\theta-\theta^2
\end{pmatrix}.
\]
In general, we study just-defined or over-defined cases when $m=q$ and $m > q$.
When the convex hull of $\{g(x_i; \theta)\}_{i=1}^{n}$ contains {\bf 0}, 
we say the convexity condition holds.
When convexity condition holds,
the solution of the optimization problem in the profile EL is given by
$$
\hat{p}_i=\frac{1}{n[1+ \hat \lambda_n^{\tau}g(x_i;\theta)]},
$$
where $\hat \lambda_n$ is the Lagrange multiplier to the
optimization problem and it satisfies
\begin{equation}
\label{Lag.eqn}
\sum_{i=1}^{n} \frac{g(x_i;\theta)}{1+ \lambda^{\tau} g(x_i;\theta)}
=0.
\end{equation}
Therefore, the profile log EL function can also be written as
$$
\el _n(\theta)=\sum_{i=1}^{n}\log(\hat{p}_i)
= - \sum_{i=1}^{n}\log[1+ \hat \lambda^{\tau}g(x_i;\theta)] - n\log n.
$$
Following \cite{QinLawless1994}, we define a profile empirical
likelihood ratio (PELR)  function
$$
W_n(\theta)
= - \sum_{i=1}^{n}\log(n \hat{p}_i)
= \sum_{i=1}^{n}\log[1+ \hat \lambda^{\tau}g(x_i;\theta)].
$$
Equivalently, the MELE of $\theta$ is given by
\begin{equation}
 \hat{\theta}_n = \arg\inf_{\theta \in\bTheta} W_n(\theta).
\end{equation}
We temporarily assume that the global minimum of the PELR is unique.
When the convexity condition does not hold, $\sum_{i=1}^{n}p_i g(x_i; \theta)=0$ 
does not have a solution and we let $\el(\theta) = -\infty$
and $W_n(\theta) = \infty$.
Under some moment conditions on $g(x; \theta)$, the convex condition holds
with probability approaching 1 at the true $\theta$ value as $n \to \infty$. 
We can also use the adjusted EL (AEL) of  \cite{Asokan} to avoid  the convexity issue. 
To avoid losing focus, we do not introduce the AEL here.

Is the MELE $\hat \theta_n$ consistent? \cite{QinLawless1994} give a Cramer style
positive answer. They show that $W_n(\theta) > W_n(\theta^*)$ 
uniformly for $\theta$ on the boundary of the compact set 
$\{\theta: \norm{\theta-\theta^*} \leq n^{-1/3}\}$ almost surely,
where $\theta^*$ is the true value of the parameter.
Consequently, when $W_n(\theta)$ is smooth, there must be a
local minimum within $O(n^{-1/3})$ neighbourhood of $\theta^*$.
If this one is taken as $\hat \theta_n$, then $\hat \theta_n$ is
strongly consistent and asymptotically normal.
These results are satisfactory for most purposes and nicely match
the properties of the parametric MLE in \cite{Cramer1946}. 
When $W_n(\theta)$ is always convex, there is a unique local minimum. 
Hence the MELE is this minimum and is strongly consistent and asymptotically normal. 

The PELR function, however, may have multiple global
minimum solutions as evidenced by the many examples to be given soon in later sections.
We use Cauchy model with score function
\begin{equation}
\label{Cauchy.score}
g(x; \theta) = \frac{ x - \theta}{ 1 + (x - \theta)^2}
\end{equation}
for illustration.
Given a random sample from Cauchy with $\theta^* = 0$, 
the score-based estimating equation
\begin{equation}
\label{cauchy.eq}
\sum_{i=1}^n g(x_i; \theta) =0
\end{equation}
has multiple solutions.
The profile log EL function defined by \eqref{EL.f} 
$\el_n(\hat \theta) = - n \log n = \sup_\theta  \el_n(\theta)$
whenever $\hat \theta$ solves \eqref{cauchy.eq}.
Hence, $\el_n(\theta)$ has multiple global maxima.
In addition, it is known that a positive proportion of the solutions
are outside any compact neighbourhood of $\theta^*$ even when $n \to \infty$.
Therefore, some global maxima of $\el_n(\theta)$
are not consistent for $\theta$ in general.

Under the parametric Cauchy model, one can consistently 
estimate $\theta$ by choosing among the solutions
the one that gives the highest parametric likelihood value.
The profile log EL function, however,  cannot be used this way 
in the Cauchy example based on the above discussion.
This leads to an important research problem: what properties
must the set of estimating functions have to ensure that
the global maximum (of the EL) is asymptotically unique and consistent?

We show that the global maximum of EL, or the global minimum of PELR,
is asymptotically unique and consistent under the following conditions:
\begin{itemize}
\item[(C1)] 
$\bbE[g(X;\theta)]=0$ has a unique solution $\theta^* \in \bTheta$.

\vspace{.5em}
\item[(C2)] 
The matrix $\bbE[g(X;\theta)g^{\tau}(X;\theta)]$ is finite and positive definite,
and $\bbE\{ \| g(X;\theta)\|^3\} < \infty$ for any $\theta \in \bTheta$.

\vspace{.5em}
\item[(C3)] 
For any compact set $C \subseteq \bTheta$ such that $g(x;\theta)$ 
is Lipschitz continuous in $C$: there exists a function $G$ such that 
\[
\norm{g(x;\theta') - g(x;\theta)} \leq G(x)\norm{\theta'-\theta}
\]
with $\bbE\{ G^2(X)\} < \infty$ for all $\theta, \theta' \in C$.

\vspace{.5em}
\item[(C4)] 
The parameter space $\bTheta$ is a closed subset of ${\mathbf R}^q$.

\vspace{.5em}
\item[(C5)] 
There exists a continuous scaling vector $b(\theta)$ such that
$h(X;\theta)=b^{\tau}(\theta)g(X; \theta)$ satisfies
\vspace{.3em}
\begin{enumerate}
\setlength\itemsep{1em}
\item[(i)]
$\bbE\{\limsup_{\norm{\theta}\to\infty}\norm{h(X;\theta)}^2\}<\infty.$

\item[(ii)]
$\liminf_{\norm{\theta}\to\infty}\norm{\bbE\{h(X;\theta)\}}>0.$

\item[(iii)] $\bbE\big[\lim_{r\to\infty}\Delta_h(X;r)\big]=0$ where
\[
\Delta_h(X; r)
=\sup_{\norm{\theta},\norm{\theta'}\geq r}
\norm{h(X;\theta)-h(X;\theta')}.
\]
\end{enumerate}
\end{itemize}


These conditions require the estimating functions being
smooth, having finite moments, properly identifying the parameter value, 
as well as they do not degenerate as $\|\theta\| \to \infty$.
They are non-restrictive and intuitively necessary for consistent
estimation of $\theta$.
\cite{Huber} and \cite{QinLawless1994} 
placed similar conditions.
We point out here the Cauchy example does not satisfy C5. 

We first give a few preparatory lemmas before our main theoretical result.
We assume the general setting and notations already introduced.

\begin{lemma}
\label{lemma.owen}
Suppose $Y_1, \ldots, Y_n$ are independently and identically
distributed random variables with finite second moment.
Then $\max_i \{ |Y_i| \} =  o ( n^{1/2})$.
\end{lemma}
This result is Lemma 11.1 in  \cite{Owen2001}.
Applying this lemma to $Y_i = g(X_i; \theta)$ under condition
(C2), we learn that $\max_i \norm{g(X_i; \theta)}= o(n^{1/3})$.
This makes $\log \{ 1+ \lambda^\tau g(x_i; \theta)\}$ well-defined
if $\lambda = O(n^{-1/3})$. Note that $o(1)$ is an order valid
in the mode of almost surely as compared to $o_p(1)$ which is
in the mode of in probability.

\begin{lemma}
\label{as.order}
Suppose C1 and C2 hold. We have $W_n(\theta^*)<n^{1/4}$ almost surely. 
\end{lemma}

Under C1 and C2, it is well known that $W_n(\theta^*)$ has chisquare limiting
distribution as $n \to \infty$. Hence,  $W_n(\theta^*) = O_p(1)$.
Because we aim to prove global consistency in the mode of almost surely,
we need a bound in the same mode.
The proof is very simple because we only need a crude upper bound.

\begin{proof}
Recall that $W_n(\theta^*) = \sum_{i=1}^n \log \{ 1 + \hat \lambda_n^\tau g(X_i; \theta^*)\}$
with $\hat \lambda_n$ being the solution to Lagrange equation \eqref{Lag.eqn}.
When $g(X; \theta^*)$ has zero mean and finite and positive definite variance,
it is well known that $\hat \lambda_n = O_p(n^{-1/3})$. Using the same
proof of \cite{Owen2001}, we can obtain a stronger conclusion:
 $\hat \lambda_n = O(n^{-1/3})$. 
 From the elementary inequality of $\log ( 1 + x) \leq x$
 and the law of iterative logarithm \citep{Vaart}, we find
 \[
 W_n(\theta^*) \leq \hat \lambda_n^\tau \sum_{i=1}^n g(X_i; \theta^*)
 = O(n^{-1/3}) \times O( \{n \log \log n\}^{1/2}) = o (n^{1/4}).
 \]
 This completes the proof.
 \end{proof}

This intermediate result shows that with probability going to 1, 
$W_n(\theta^*)$ is bounded by a sequence going to $\infty$ at a rate of $n^{1/4}.$

\begin{lemma}
\label{lemma.asokan}
Suppose C1-C2 hold and let $\theta \neq \theta^*$ be a fixed parameter value.
Then
\[
W_n(\theta) > W_n(\theta^*)
\]
almost surely.

\end{lemma}

This conclusion has appeared in \cite{Asokan}.
It strongly suggests that under conditions C1-C2 the MELE  is globally consistent. 
A crucial missing piece of this lemma is the uniformity in $\theta$.
Nevertheless, this lemma serves as an important step-stone in our proof of the global consistency. 
For this reason, we provide a proof in this paper.

\begin{proof}
Let $\bar{g}_n(\theta) = n^{-1} \sum_{i=1}^{n} g(x_i;\theta)$,
and $\widehat \lambda (\theta) = n^{-3/4} \bar g_n(\theta)$.
When the convex condition holds, the PELR function
\[
W_n(\theta) =  \sum_{i=1}^{n}\log[1+\lambda^{\tau} g(x_i; \theta)]
\]
with $\lambda$ being the solution to \eqref{Lag.eqn}.
If the convex condition does not hold, $W_n(\theta)= \infty$. We introduce a function of $\lambda$ given $\theta$:
$$
\Omega_n(\lambda; \theta) 
= \sum_{i=1}^{n}\log \{ 1+\lambda^{\tau}g(x_i;\theta)\}.
$$
The domain of this function is
\[
\mathbb{D}_n(\theta) 
=
\big \{
\lambda: 1+\lambda^{\tau} g(x_i;\theta)> 0,~~~ \forall ~ i \in \{1, 2, \ldots, n\}
\big \}
\]
Because ${\bf 0} \in \mathbb{D}_n(\theta)$, the domain is not empty.
When the convex condition holds, there is a unique 
$\widetilde \lambda(\theta) \in \mathbb{D}_n(\theta)$ that solves \eqref{Lag.eqn}. 
Comparing their expressions, we easily find that
$\Omega_n(\widetilde \lambda(\theta) ;\theta) = W_n(\theta)$.

At the same time, the second derivative matrix
\[
\frac{\partial^2 \Omega_n(\lambda;\theta)}{\partial\lambda\partial\lambda^{\tau}}
= 
- \sum_{i=1}^{n}\frac{g(x_i;\theta)g^{\tau}(x_i;\theta)}{[1+\lambda^{\tau} g(x_i;\theta)]^2}
\]
is at least semi-negative definite, but mostly negative definite.
We harmlessly assume the negative definiteness here.
Therefore, $\tilde \lambda(\theta)$  is the unique maximum point of $\Omega(\lambda; \theta)$. 
In particular, we have
\[
W_n(\theta) = \Omega_n(\tilde \lambda(\theta); \theta)  
\geq 
\Omega_n(\widehat \lambda(\theta); \theta).
\]
for any $\theta$.  
Because $\widehat \lambda(\theta) = O(n^{-3/4})$ by design
and $\max_i \|g(x_i; \theta)\|  = o(n^{1/2})$, we have 
$\widehat \lambda(\theta) g(x_i; \theta) = o(1)$ uniformly in $i$. 
So $\Omega_n(\widehat \lambda(\theta); \theta)$ is well-defined
almost surely.
This allows us to apply Taylor's expansion to
\[
\Omega_n(\widehat \lambda(\theta); \theta)
= 
 \sum_{i=1}^{n}\log \{ 1+ \widehat \lambda^{\tau}(\theta) g(x_i;\theta)\}.
 \]
The expansion leads to $\Omega_n(\widehat \lambda(\theta); \theta) \to \infty$ at a rate of at least $n^{1/4}$ almost surely.
Therefore we also have $W_n(\theta) \to \infty$ at a rate of at least $n^{1/4}$ almost surely.
 We have omitted some trivial details here. 
 Because $W_n(\theta^*) =o(n^{1/4})$ by Lemma 
 \ref{as.order}, we must have $W_n(\theta) > W_n(\theta^*)$ almost surely. 
\end{proof}

Simplistically, we must also show that the inequality holds uniformly in $\theta$ 
outside any small neighbourhood of $\theta^*$, just like the proof of \cite{Wald1949}.
In particular and ideally, we must handle non-compact $\Theta$.
We borrow much insight from \cite{Huber}.
To enhance readability, we list many intermediate results in a lemma below. The detailed proofs are given in the appendix.

\begin{lemma}
\label{lemma2.3}
Suppose C1-C5 hold and let $\theta \neq \theta^*$ be a fixed parameter value.
For a constant $\rho > 0$, let
$B(\theta, \rho) = \{ \theta': \norm{\theta' - \theta} < \rho\}$.
Let $\widehat \lambda_n (\theta) = n^{-3/4} \bar{g}_n(\theta)$.
Then, for all sufficiently small positive value $\rho$ which may depend on
$\theta$, as $n \to \infty$,
\vspace{.5em}
\begin{itemize}
\item[(a)]
$\sup_{\theta'\in B(\theta, \rho)} \norm{\widehat{\lambda}(\theta)} = O(n^{-3/4})$.

\vspace{.5em}
\item[(b)]
$\max_{1 \leq i \leq n}
\sup_{\theta' \in B(\theta, \rho)} \norm{g(x_i; \theta')} =o(n^{1/3})$.

\vspace{.5em}
\item[(c)]
$
\sup_{\theta' \in B(\theta,\rho)} \sum_{i=1}^{n} \norm{g(x_i;\theta') }^2
=
O(n).
$
\end{itemize}

Let $W_n( \theta; \rho) = \inf_{\theta' \in B(\theta, \rho) } W_n(\theta')$
and $\overline{W}_n(r) =  \inf_{\| \theta\| > r } W_n(\theta)$.
Then, for all sufficiently small positive value $\rho$ which may depend on
$\theta$, and sufficiently large $r$,  as $n \to \infty$,
\begin{itemize}
\vspace{.5em}
\item[(d)]
$W_n( \theta; \rho) > W_n(\theta^*)$ almost surely.

\vspace{.5em}
\item[(e)]
$\overline{W}_n(r) > W_n(\theta^*)$ almost surely.
\end{itemize}
\end{lemma}

Now we state the main theorem. 
\begin{theorem}
\label{global.consistency}
Assume the general setting and notation introduced previously.
Under conditions C1-C5, $\hat{\theta}_n {~\to~ } \theta^*$
almost surely as $n \to \infty$. 
\end{theorem}

\begin{proof}
By Lemma \ref{lemma2.3} (d), there is a sufficiently large $r > 0$ such that 
$\overline{W}_n(r) > W_n(\theta^*)$ almost surely.

For any $\epsilon>0$, let
\begin{equation}
\label{Theta.r.epsilon}
\overline{\Theta}_{r, \epsilon} 
=
\{ \theta: \norm{\theta - \theta^*} \geq \epsilon, \norm{\theta} \leq r \}.
\end{equation}
By Lemma \ref{lemma2.3}(c),  for any $\theta \in \overline{\Theta}_{r, \epsilon} $, 
there is an associated sufficiently small $\rho_{\theta}$ such that 
\[
W_n(\theta; \rho_{\theta}) > W_n(\theta^*)
\]
almost surely.
Under condition C4, $\overline{\Theta}_{r, \epsilon}$ is compact. 
Hence, there exist a finite number of $\theta_1, \ldots, \theta_K$
such that 
\[
\overline{\Theta}_{r, \epsilon} \subset \cup_{j=1}^{K} B(\theta_j, \rho_{\theta_j}).
\]
Therefore, 
\[
\inf_{\theta \not \in B(\theta^*, \epsilon)} W_n(\theta)
\geq
\min \Big \{ 
\overline{W}_n(r), 
\min_{1 \leq j \leq K} \inf_{\theta \in B(\theta_j, \rho_{\theta_j}) } W_n(\theta) 
\Big \}
> W_n(\theta^*)
\]
almost surely.
This implies the global minimum of $W_n(\theta)$, 
or the global maximum of $\el(\theta)$
is within the $\epsilon$-neighborhood of $\theta^*$.
Because $\epsilon$ can be chosen arbitrarily small, the
MELE is globally consistent.
\end{proof}

\section{Test for Global Maximum}
In most applications, we trust that the consistency conclusion of \cite{QinLawless1994}
is applicable to the ``MELE'' at hand without due diligence.
If in doubt, one should first confirm that C1-C5 hold 
and subsequently use the global maximum of the likelihood as the MELE for statistical analysis.
This approach, however, becomes practical only if we can determine whether 
a local maximum at hand is in fact global.
The generic problem of determining whether a local maximum is
global is a well-known challenging problem, or
one would be instantly famous with a powerful general solution.
For special cases, some approaches have been
discussed in the literature such as \cite{DeHaan}, \cite{Veall}, 
and \cite{Jiang1999}. 

In this paper, we utilize the global consistency conclusion
and give a powerful solution.
To better appreciate the novelty of the proposed approach, we first
briefly describe the methods of \cite{Veall} and \cite{Jiang1999}.

Suppose one wishes to maximize a function $\ell(\theta)$ over a bounded
region $\Theta$, and $\hat \theta$ is the current best solution. 
\cite{Veall} propose to draw an iid sample $\theta_1, \ldots, \theta_n$
for some $n$ uniformly from $\Theta$. 
Compute $L_i = \ell(\theta_i)$ for $i=1, \ldots, n$.
Making use of extreme value theory of \cite{DeHaan},
one can construct a level $p$ confidence interval for $\sup_\Theta \ell(\theta)$
in the form of $(L_{(n)}, L^p)$ with
\begin{equation}
\label{DeHaan.ci}
L^p = L_{(n)} + (L_{(n)} - L_{(n-1)})/\{ p^{-2/q} -1\}
\end{equation}
where $L_{(n)} $ and $L_{(n-1)}$ are two largest order statistics of $L_i: i=1, \ldots, n$
and $q$ is the dimension of $\theta$.
One rejects the hypothesis that $\hat \theta$ is the global maximum if
$\ell(\hat \theta) < L^p$.
Setting $n=500,$ \cite{Veall} experimented this method on 5 examples and showed
such a test has reasonably close to $100p\%$ rejection rate when $\hat \theta$ is a global
maximum, and respectably high reject rate when it is not.
As time goes, both computing power and statistical theory of this approach 
become outdated.
Yet this approach remains intelligently a good reference point.

\cite{Jiang1999} mainly targets the global maximum problem under a parametric model.
Suppose we have an iid sample from a distribution with a parametric density $f(x;\theta)$. 
Let $\hat{\theta}_n$ be a local maximum of the log parametric likelihood $\ell_n(\theta)$.
Let
\begin{equation}
\label{Bartlett}
\varphi(x;\theta) =
\Big(\frac{\partial}{\partial\theta}\log f(x;\theta)\Big)^2
+\frac{\partial^2}{\partial\theta^2}\log f(x;\theta).
\end{equation}
When the parametric model is regular and correctly specified,
by the Bartlett identity \citep{Bartlett1953a,Bartlett1953b}, we have
$$
\bbE \{ \varphi(X; \theta^*) \} = 0.
$$
Hence if $\hat{\theta}$ is a consistent estimator, one anticipates
\[
\overline{\varphi}_n = \frac{1}{n}\sum_{i=1}^{n}\varphi(x_i; \hat{\theta}) \approx 0.
\]
Consequently, \cite{Jiang1999} proposes to reject the hypothesis that 
$\hat{\theta}$ is a global maximum (or a consistent estimator) when
the norm of $\overline{\varphi}_n$ exceeds some threshold value,
which is usually chosen based on the limiting distribution of $\overline{\varphi}_n$.
Clearly, their approach also has the consistency and global maximum 
problems linked together. 
It requires a proper norm and a satisfactory scaling factor for $\overline{\varphi}_n$.
Both are conceptually simple but can lead to many tedious issues
in implementation. Its effectiveness can suffer from the risk
of model misspecification.

In the context of EL, we invent an elegant alternative.
Suppose conditions C1-C5 hold. 
If $\hat \theta_n$ is a global maximum, then
by Theorem \ref{global.consistency}, $\hat \theta_n$ is consistent.
By \cite{QinLawless1994}, we anticipate that when $q > m$,
\[
2W_n(\hat \theta_n) \to \chi^2_{q-m}.
\]
If $\hat \theta_n$ is not a global maximum, then $2W_n(\hat \theta_n)$
is stochastically inflated. This observation leads to a simple approach. 

\vspace{1em}
\noindent
{\sc Global maximum test}.
{\it
Given a local maximum $\tilde{\theta}_n$ of $\el(\theta)$, 
we reject the claim that it is a global maximum when
\[
p\mbox{-value} = P \{ \chi^2_{q-m} >  2W_n(\tilde \theta_n)  \} < \alpha
\]
for some pre-selected $\alpha \in (0, 1)$, where $\chi^2_{q-m}$ stands
for a chi-square random variable with $q-m$ degrees of freedom.
}

\vspace{1em}
\noindent
{\sc Remark}. Technically, we have not proved that the limiting distribution of $W_n(\hat \theta_n)$,
when $\hat \theta_n$ is the global maximum, is the chisquare given by
\cite{QinLawless1994}. However, our test is always valid and
conservative.

In applications, the larger the value of $2W_n(\tilde \theta_n)$ is,
the more effort should be made to search for a potential and yet to be
located global maximum of the likelihood. 
The above $p$-value provides a standardized metric for judgement.
Although this procedure resembles the statistical significance test,
we do not actually test for any hypotheses about the model.
Hence, we need not adopt the convention $\alpha = 0.05$.
One may want to be much more liberal in searching for ``true''
global maximum.

Note that when the model is just defined, that is, $m=q$, we often have at least
one solution to the estimating function. All solutions are global maximum
of the profile EL $\el(\theta)$. Hence, the problem of identifying the
global maximum does not exist, though not all of them are consistent.

\section{Recipe for global consistency}
As we have already pointed out, in some cases, a global maximum
is not necessarily a consistency estimator.
This is particularly true in the EL context when $m=q$, 
the number of equations is the same as the number of parameters.
Consider the Cauchy distribution example of \cite{Small}.
Suppose $\tilde \theta_1, \tilde \theta_2, \ldots, \tilde \theta_K$ 
are solutions to $\sum  g(x_i; \theta) = 0$.
Given $\theta = \tilde \theta_k$, the solution to the profile EL is given by
$p_i = 1/n$. Therefore, we have
\[
\el(\tilde \theta_k) = \sup_\theta \el(\theta).
\]
Namely, all of them are global maximum of the profile EL and
we cannot choose a consistent estimator accordingly.

Why does Theorem \ref{global.consistency} not apply? 
The key is the failure of condition C5, which requires the norm of the estimating function
(vector valued in general) does not shrink to ${\bf 0}$ while the variance remains bounded
as $\|\theta\| \to \infty$. 
This is apparently not true for the Cauchy score function given in \eqref{Cauchy.score}.

Under the parametric Cauchy model, one would not completely rely on estimating
functions for statistical inference. The parametric likelihood can be reliably used to
choose a consistent estimator among 
$\tilde \theta_1, \tilde \theta_2, \ldots, \tilde \theta_K$.
In addition, we need not exhaust all solutions of the estimating equation, but focus
on those close to the sample median.

It is widely accepted that the likelihood function contains all information in the data
needed for statistical inference if the model assumption is solid. The EL approach
we adopted gives up all parametric model information except for an estimation function.
Suppose in addition to this estimating function, we also know that the population
distribution is symmetric. We can then feed the EL with an expanded
set of estimating functions:
\begin{equation}
\label{Cauchy.eq2}
g(x; \theta) = \left (
\begin{array}{c}
 (x - \theta)/\{1+ (x- \theta)^2\} \\
 (x - \theta)^{1/3}
 \end{array}
 \right ).
\end{equation}
 The second entrance of the above estimating function validates C5 so that the global
 maximum of the new profile EL function is consistent.
According to \cite{QinLawless1994}, the MELE is an optimal estimator
of $\theta$ if $g$ is the score function.
Adding unbiased estimating functions
to the score function does not lead to more efficient estimator.
However, it helps in achieving global consistency and dispeling nuisance
local maxima.

 \vspace{1em}
 \noindent
 {\sc Global consistency remedy}.
 {\it
 When the profile log EL function $\el(\theta)$ based on unbiased
 estimating function $g_1(x; \theta)$ does not satisfy any of C1-C5,
 search for additional unbiased estimating function $g_2(x; \theta)$ so that
\[
g(x; \theta) = \left ( 
\begin{array}{c}
g_1(x; \theta)\\
g_2(x; \theta)
\end{array}
\right )
\]
satisfies condition C1-C5.}
 
We now apply this remedy to several well discussed problems in the literature.

\subsection{Non-linear Regression}

Non-linear methods are popular in econometrics. 
We often choose to estimate the model parameter
as the maximum of some objective function \citep{Veall}.
Unlike linear models, the solution to the optimization problem
usually does not have a closed/analytical form.
Econometricians generally employ a numerical algorithm
to search for parameter values at which the gradient of the objective function is 0.

When the objective function has multiple local maxima, the resulting estimator can be inconsistent. 
Because of this, \cite{DeHaan} proposes to test whether or not a given local maximum 
of the objective function is global maximum based on asymptotic extreme value theory. 
As seen earlier, his approach requires evaluating the objective function at a large number
of parameter values. 
These extra computational cost of \cite{DeHaan} can hence be prohibitive.
Designing a proper sampling scheme when the parameter space is of high dimension can also be an obstacle to its generalization.
In comparison, our approach is simpler and much more powerful.

We illustrate the proposed method through a non-linear regression example of
\cite{Veall}.
Suppose $n$ paired independent observations 
$(x_1, y_1), \ldots, (x_n, y_n)$ are described by the following regression model
\begin{equation}
\label{regr.model}
y_i = \theta + \theta^2 x_i+\epsilon_i,
\end{equation}
with real valued parameter $\theta$, and error terms
$\epsilon_i$ being iid with $N(0, \sigma^2)$ distribution.
 
It is natural to estimate $\theta$ through minimizing the sum of squared errors.
The sum of squares objective function is defined to be
$$ 
S_n(\theta)=\sum_{i=1}^{n}(y_i - \theta - \theta^2 x_i)^2.
$$
The least square estimate is conceptually defined by
$$
\hat{\theta}=\arg \min_{\theta \in\mathbb{R}} S_n(\theta).
$$
We usually solve the equation $\partial S_n(\theta)/\partial \theta =0$ 
to obtain $\hat{\theta}$. 
Hence $\hat{\theta}$ is taken as a solution to
\[
\sum_{i=1}^{n} (1+2 \theta x_i)(y_i - \theta - \theta^2 x_i ) = 0.
\]
Being cubic in $\theta$, this equation can have up to three real roots.
A diligent applicant will locate all real roots and select the global minimum
of $S_n(\theta)$ as the estimate of $\theta$. 

When the sample size is very large, all real solutions may converge to
the true value. In addition, there is a simple approach to resolving the 
choice among potentially three roots.
This paper uses this model to provide a convenient and well-investigated
example to illustrate the effectiveness of the proposed
{\sc global consistency remedy} as well as the {\sc global maximum test}.

Suppose we wish to apply the EL approaches to data analysis under this
regression model without the parametric error distribution assumption. 
We may hence work with the profile EL function based on the estimating function
$$
g_1(x, y; \theta) =  (1+2 \theta x)(y - \theta - \theta^2 x ).
$$	
Further, we find that $\bbE\{ g_1(X, Y; \theta)\} =0$  has 3 real roots 
when the true value $\theta^*=1$, $\text{var}(X) = 1$, and $\bbE(X)$ is sufficiently large.
Hence, condition C1 does not hold in general.

At the same time, if  $\sum g_1(x_i, y_i; \theta)=0$ has three roots, then they
are all global maxima of the profile EL function.
Due to the violation of condition C1, not all of them are consistent.
The asymptotic results of \cite{QinLawless1994} falters.
Our proposed test for global maximum is inapplicable.

However, the remedy for global consistency is useful.
Assume the regression model \eqref{regr.model} and the
finite second moment of the error distribution.
It is in fact as a restricted linear regression model.
This readily leads to another unbiased estimating function
\[
g_2(x, y; \theta) = x(y - \theta - \theta^2 x).
\]
Combining $g_1(x, y; \theta)$ and $g_2(x, y; \theta)$ and after
simplification, we may use the profile EL based on 2-dimensional estimating
function
\begin{equation}
\label{regr.g}
g(x , y; \theta )
=
\begin{pmatrix}
  y - \theta - \theta^2 x  \\
x (y - \theta - \theta^2 x)
\end{pmatrix}.
\end{equation}
Note that we have replaced $g_1$ by $g_1 - 2 \theta g_2$.
The model is now over-defined with $q=1 < m=2$.
The PELR function has form
\[
W_n(\theta)=-\sum_{i=1}^{n}\log[1+\tilde{\lambda}^{\tau}g(x_i, y_i; \theta)],
\]
with $\tilde{\lambda}$ being vector valued Lagrange multiplier. 
The MELE of $\theta$ is the global maximum of $W_n(\theta)$
and it is consistent.

We next verify Theorem \ref{global.consistency}
conditions C1-C5 for $g(x, y; \theta)$ in \eqref{regr.g} under additional but generic
conditions: (1) $0 < \bbE(\epsilon^4)<\infty$, (2)  $0 < \bbE(X^4)<\infty$,
and $\epsilon$ and $X$ are independent.

First, the current $\bbE\{ g(X, Y; \theta)\}$ is the partial
derivative of $\bbE\{ (Y - \beta_0 - \beta_1 X)^2\}$ at
$\beta_0 = \theta$ and $\beta_1 = \theta^2$.
Because  $\bbE\{ (Y - \beta_0 - \beta_1 X)^2\}$ has unique
minimum, the solution to $\bbE\{ g(X, Y; \theta)\}=0$ must be unique.
Condition C2 is obviously satisfied under moment assumptions.
Further, we have
\[
\frac{\partial g(x, y; \theta)}{\partial \theta}
=
\begin{pmatrix}
- (1 + 2 \theta x)\\
- x(1 + 2\theta x )
\end{pmatrix}.
\]
Over any compact parametric set such as $| \theta | > M$, 
it is seen
\[
 \norm{\partial g(X,Y;\theta)/{\partial \theta}}  \leq  (1+|x|) ( 1 + 2 M |x| ).
 \]
Clearly $\bbE \{ (1+|X|) ( 1 + 2 M |X| )\}^2 <\infty$. 
This shows that C3 is satisfied.

Condition C4 is obviously true because the space of $\theta$, ${\bf R}$
is a close set.

Finally, let $b(\theta) = 1 +  \theta^2$.
It is seen that
\[
\lim_{|\theta |\to\infty} \frac{g(X,Y; \theta)}{b(\theta)}
=
\begin{pmatrix}
- X\\
-X^2
\end{pmatrix},
\]
which has finite expectation. 
As the limit does not depend on $\theta$, we
get
$\lim_{r\to\infty} \Delta_h(X,Y; r)=0$.
Therefore, we have verified (i) and (iii) of (C5).
Finally, we have
$$
\bbE[g(X,Y; \theta)]
=
\begin{pmatrix}
\bbE(Y) - \theta - \theta^2 \bbE(X) \\
\bbE(XY) - \theta \bbE(X) - \theta ^2\bbE(X^2)
\end{pmatrix}.
$$
Hence 
$$
\lim_{|\theta| \to\infty}
\frac{\norm{\bbE\{ g(X,Y; \theta)\} }}{b(\theta )}> 0
$$
and (ii) of C5 holds. 
By Theorem \ref{global.consistency},
the MELE $\hat{\theta}$ is globally consistent.

\subsection{Curved Exponential Family}
Exponential family enjoys many nice statistical properties for estimation and inference
that are not shared by non-exponential families \citep{Efron2022}.
In addition, \cite{Efron1975, Efron1978} study a subset of the exponential family, 
which is referred to as curved exponential family.  
\cite{Efron1975} uses curvature to quantify how nearly `exponential' a curved exponential family is. 
When the curvature of the curved exponential family is small, 
the properties of exponential family are generally retained. 

The score equation for a curved exponential family often has multiple roots,
see \cite{Barndorff, Sundberg}. 
However, locating the consistent global maximum is not always straightforward 
in the parametric likelihood context. 
Our remedy for global consistency provides an ideal solution to this problem.
We use a simple example to make this point.

Suppose we have iid observations $x_1, \ldots, x_n$ from a distribution function $F$.
Instead of directly assuming that $F$ is a member the curved exponential family $N(\theta, \theta^2)$,
one may work with a nonparametric $F$ possessing a
well-motivated unbiased estimating function.
The score function of $N(\theta, \theta^2)$ in this case is an
unbiased estimating function:
\[
 g_1(x;\theta) =\frac{x^2}{\theta^3}-\frac{x}{\theta^2}-\frac{1}{\theta}. 
\]
The corresponding estimating equation is quadratic in $\theta$ and
which has two real roots in most cases. 
In theory, $|\theta| = \infty$ is also a solution.
Similar to the nonlinear model example, one can identify the
global maximum of the parametric likelihood, if assumed.

If we activate the profile EL function built on $g_1(x; \theta)$,
it will attain the global maximum at both two roots, excluding
$\theta = \pm \infty$. Clearly, our global consistency conclusion
does not apply because $g_1(X;\theta)$ does not satisfy conditions C1-C5. 
However, our results work differently.
It is seen that the curved exponential family is a sub-family of a regular
exponential family, which comes with a set of regular score functions.
A systematic approach is then activating the profile EL function built on
estimating functions made of regular score functions together with
a parametric structure. 
For this particular model, the resulting (vector-valued) estimating function
is simply
\begin{equation}
\label{curve.funct}
g(x; \theta)
=
\begin{pmatrix}
x - \theta \\
x^2 - 2 \theta^2
\end{pmatrix}.
\end{equation}
Note this approach does not work under the parametric likelihood as
the set of estimating functions does not have a solution in general.

More generally, the curved exponential family has density function
(with respect to some $\sigma$-finite measure) in the form
$$
f(x;\theta)=\exp\big [ \sum_{j=1}^{m} q_j(\theta)T_j(x) - h(\theta)\big ]
$$
for some functions $q_j(\theta)$, statistics $T_j(x)$, and normalization constant
\[
h(\theta) = 
\log \Big [ \int \exp\big \{\sum_{j=1}^{m}q_j(\theta)T_j(x) \big \} dx \Big ].
\]
In general, $\dim(\theta) < m$ for curved exponential family. 
See \cite{Lehmann1998}.
Its score function is given by
\[
s(x; \theta) 
=
\sum_{j=1}^{k} [q_j'(\theta)T_j(x) ]-   h'( \theta)
\]
where $h'(\theta)$ denotes the derivative.

Let $\eta_j = q_j(\theta)$ for $j=1, \ldots, m$ and $\eta$ for the vector form. 
Then, we must have $h(\theta) = \tilde h(\eta)$ for some smooth $\tilde h(\eta)$.
Regarding $\eta_j$ as algebraically independent parameters, we obtain
score function for the corresponding regular exponential function:
\[
\tilde s(x; \eta)  = T(x) -   \tilde h'(\eta).
\]
Combining estimating function $\tilde s(x; \eta(\theta))$ with the EL, we
obtain the corresponding MELE $\hat{\theta}_n$. 
It is straightforward to verify that $\tilde s(x; \eta(\theta))$ satisfies C1-C5 in most cases. 
Hence it is globally consistent in general. 

Clearly, our proposed global consistency remedy and the global maximum test
provide powerful solutions to the problems under the curved exponential family.

\subsection{Mixture Model}
Finite mixture models are widely used in various applications 
       \citep{Titterington1985,Mclachlan2019}. 
The MLE is usually the choice of estimation and is often obtained through the popular EM-algorithm
\citep{Dempster1977}.
The MLE is generally consistent \citep{Chen2017}
and the EM-algorithm is guaranteed to converge to a local maximum \citep{Wu1983}.
Yet the multiple root problem, or local maximum issue, persists. 

Let $\phi(\cdot)$ be the density function of the standard normal.
Consider a special two-component normal mixture
model discussed in \cite{Jiang1999} whose density function is given by
\[
f(x; \theta)
=
\frac{\pi}{\sigma_1} \phi(\frac{x-\theta}{\sigma_1})
+
\frac{1-\pi}{\sigma_2}\phi(\frac{x-\mu_2}{\sigma_2}).
\]
We only regard  the first subpopulation mean $\theta$ as an unknown parameter. 
Assume the values of mixing proportion $\pi$, the subpopulation mean $\mu_2$,
and subpopulation variances $\sigma_1^2$ and $\sigma_2^2$ are known.

The score function of this normal mixture model is given by
\[
s(x; \theta) = \frac{(x-\theta)\phi( ({x-\theta})/{\sigma_1})}{f(x;\theta)}
\]
and it tends to have multiple roots 
when $\theta$ and $\mu_2$ are far apart \citep{Titterington1985}.
Once a root of the score function is located, the question of whether the root corresponds to the global maximum arises. 
\cite{Jiang1999} propose a recipe and use this model for illustration.
We remark that this special model is regular in the sense of \cite{Lehmann1998}.
The Bartlett identity \eqref{Bartlett} in this case leads to
\[
\varphi (x; \theta)
= \frac{[(x-\theta)^2-\sigma_1^2]\phi( ({x-\theta})/{\sigma_1})}{f(x;\theta)}
\]
which is an unbiased estimating function. Suppose $\hat \theta_n$ is the MLE
and global maximum of the likelihood, then $\hat \theta_n$ is asymptotically
normal based on $n$ iid observations. Under some smoothness conditions,
there is a function $\sigma^2(\theta)$ such that 
\[
T_n(\hat \theta_n) = \frac{\sum_{i=1}^n \varphi (x_i; \hat \theta_n)}
{ \sqrt{n} \sigma (\hat{\theta}_n)} \to N(0, 1)
\]
in distribution.
When $\hat{\theta}_n$ is not the global maximum and hence not consistent, 
$|T_n(\hat \theta_n)|$ is stochastically larger.
They therefore suggest to reject the claim that $\hat \theta_n$
is a global maximum when $|T_n|$ is sufficiently large. 
Let $I(\theta)=\mbox{var}[s(x;\theta)]$ be the Fisher information
and $ \varphi'(x; \theta)$ be the derivative of $ \varphi(x; \theta)$
with respect to $\theta$. 
According to \cite{Jiang1999}, 
\[
\sigma^2(\theta)
=
\bbE^2[\varphi(X;\theta)]+
I^{-1}(\theta) \big [
2\bbE\{s(X;\theta)\varphi(X;\theta)\}
\bbE \{ \varphi'(X;\theta) \} +
\bbE^2\{\varphi'(X;\theta)\}
\big ].
\]
Clearly,  $\sigma^2(\theta)$ has a complex form.
The step of working out $\sigma^2(\theta)$ or directly estimating this
variance can lead to other complex issues in general.
In comparison, the proposed remedy for global consistency as well as 
the test for global maximum fit the occasion nicely. 
In particular, we may define the profile EL function with 
\[
g(x; \theta) 
=
\begin{pmatrix}
s(x; \theta) \\
\varphi(x; \theta) 
\end{pmatrix}
= \frac{\phi(({x-\theta})/{\sigma_1})}{f(x;\theta)}
\begin{pmatrix}
x-\theta\\ 
(x-\theta)^2-\sigma_1^2
\end{pmatrix}.
\]
Alternatively, we include three score functions in the estimating function $g(x;\theta)$: two for centres of mixture and one for the mixing proportion. 
\[
g(x;\theta)
=\begin{pmatrix}
s(x;\theta)\\s(x;\mu_2)\\s(x;p)
\end{pmatrix}
=
\begin{pmatrix}
\dfrac{(x-\theta)\phi(({x-\theta})/{\sigma_1})}{f(x;\theta)}\\
\dfrac{(x-\mu_2)\phi(({x-\mu_2})/{\sigma_2})}{f(x;\theta)}
\\
\dfrac{\sigma_2\phi((x-\theta)/\sigma_1)-\sigma_1\phi((x-\mu_2)/\sigma_2)}{f(x;\theta)}
\end{pmatrix}
\]
We regard $\mu_2$ and $\pi$ as parameters when computing the corresponding score functions
but use their known values in the above $g(x; \theta)$ and in simulations.

For both options of {\sc Global consistency remedy}, the resulting MELE is globally consistent. 
However, we find that when the sample size is small, the first option of adding $\phi(x;\theta)$ 
often results in other local minima of the EL function being close to 0. 
The second option of adding another two score functions is free from this problem.
Our approach does not require computing $\sigma^2(\theta)$.
Furthermore, the asymptotic distribution of $W_n(\hat \theta)$ is much simpler.

Judging whether or not a limit produced by EM-iteration is in fact 
a global maximum is a hugely important problem. 
We do not wish to generalize the result on this toy example hastily
but keep it as a promising future project.

\subsection{Asset Pricing}
In econometrics, the parameter of interest is often defined through estimating equation 
$\bbE[g(X;\theta)]=0,$ where $\dim(g)=m$ and $\dim(\theta)=q$. 
When $m=q$, we often solve the sample estimating equation
\[
\sum_{i=1}^{n} g(x_i;\theta)=0
\]
to estimate the unknown parameter $\theta$ \citep{Godambe}. 
However, when $m>q$, the sample estimating equation 
generally does not have a solution. 
We can use both generalized method of moments (GMM) and empirical likelihood (EL) 
for inference on the unknown parameter $\theta$ \citep{Hansen,QinLawless1994}. 

We now present an asset price example of
\cite{Hall}, \cite{Imbens1998} and \cite{Schennach}. 
The parameter of interest in their example is defined by
\begin{equation}
\label{price.eqn}
g(x; \theta) =
\begin{pmatrix}
\exp\{-0.72 - \theta(x_1+x_2)+3x_2\}-1\\
x_2 [ \exp\{-0.72 - \theta(x_1+x_2)+3x_2\}-1]
\end{pmatrix}.
\end{equation}
We consider the case when $\theta^*=3$ and
$(x_1, x_2)$ are iid sample from two independent $N(0, 0.16)$.
It is seen that 
\[
\bbE \{ g_1(X;\theta)]\}  =\exp \{-0.72+0.08 \theta^2+0.08(3-\theta)^2\}-1.
\]
Hence $\bbE[g_1(X; \theta)]=0$ has two roots at $\theta = 0$ and 3 which
violates condition C1. 
Theorem \ref{global.consistency} therefore does not apply. 
The MELE based on profile EL solely associated with $g_1(x; \theta)$
is not consistent.
Instead, we should employ the profile EL based on $g(x; \theta)$
given in \eqref{price.eqn}.
It is seen
\[
\bbE \{g_2(X;\theta)\}=0.16(3-\theta)\exp\{ -0.72+0.08\theta^2-0.08(3-\theta)^2\}
\]
and $\bbE\{g_2(X;\theta)\}=0$ has a unique root at $\theta=3$.
Consequently, $E \{g(X;\theta) \}=0$ has a unique root at $\theta=3.$ 
When $X_1, X_2$ are independent with $N(0, 0.16)$ distribution, 
we have
\[
\bbE \{ g_1^2(X;\theta)\}
=
\exp\{ 0.32\theta^2+0.32(3-\theta)^2-1.44\} - 2\exp\{ 0.08\theta^2+0.08(3-\theta)^2-0.72\} +1
\]
which is finite and also $\bbE[g_2^2(X;\theta)]<\infty$.
It is then obvious that $\bbE\{g(X;\theta)g^{\tau}(X;\theta)\}$ is positive definite
which validates condition C2 is. 
We have
\[
\frac{\partial g(x;\theta)}{\partial\theta}
=\begin{pmatrix}
-(x_1+x_2)g_1(x;\theta)\\
-x_2(x_1+x_2)g_1(x;\theta)
\end{pmatrix}.
\]
In a compact set of $|\theta|<M$, 
\[
\bbE[g_1^2(X;\theta)]
\leq 
\exp[0.64(M-1.5)^2]+2\exp[0.08(M-1.5)^2]+1.
\]
By the Cauchy-Schwartz inequality, it is easily implies
\[
\bbE \norm{\frac{\partial g(X;\theta)}{\partial\theta}}<\infty
\]
which validates C3.
Let 
\[
b(\theta)=(|\theta|+1)\exp(0.16\theta^2-0.48\theta),
\]
and $h(X;\theta)=g(X;\theta)/b(\theta)$.
We find
\[
E[h(X;\theta)]
=\begin{pmatrix}
\dfrac{1}{|\theta|+1}-\dfrac{1}{(|\theta|+1)\exp(0.16\theta^2+0.48)}\\ 
\dfrac{0.16(3-\theta)}{|\theta|+1}
\end{pmatrix}.
\]
Hence,
\[
\lim_{|\theta|\to\infty}\bbE[h(X;\theta)]
=
\begin{pmatrix}
0\\
-0.16
\end{pmatrix}
\]
which implies $\lim_{|\theta|\to\infty}\norm{\bbE[h(X;\theta)]}>0$
and validates C5 (ii).
Furthermore, we have
\begin{align*}
|h_1(x;\theta)|
&=|\frac{g_1(x;\theta)}{b(\theta)}|
\\
&\leq \frac{1}{|\theta|+1}\exp[-1.2+|\theta|(|x_1|+|x_2|)+3x_2-0.16\theta^2]
\\
&=\frac{1}{|\theta|+1}\exp[-0.16(|\theta|-\frac{25}{8}(|x_1|+|x_2|))^2+\frac{25}{16}(|x_1|+|x_2|)^2+3x_2]
\\
&\leq \frac{1}{|\theta|+1}\exp[\frac{25}{16}(|x_1|+|x_2|)^2+3x_2].
\end{align*}
Hence, $\lim_{|\theta|\to\infty}|h(x;\theta)|=0$.
Similarly, we find $\lim_{|\theta|\to\infty} |h_2(x;\theta)|=0$.
These validate C5 (i) and (iii).

In conclusion,  we have verified conditions C1-C5.
Hence the MELE based on the profile EL associated with estimating function $g(X;\theta)$
is consistent. Our proposed approaches provide general solutions to
the multiple-root problem and the problem of determining whether a local maximum
is in fact a global maximum that arise in statistical analysis of date under the asset pricing model.

\section{Numerical experiments}

We assess the effectiveness of the proposed approaches
based on examples given in previous sections. 

{\bf Cauchy model}
As already claimed, the score equation under Cauchy model 
\eqref{Cauchy.score} has multiple roots \citep{Small}. 
We first use experiment to verify that these multiple roots do not converge
as the sample size $n$ increases.
We generate a random sample of size $n$ from the standard Cauchy distribution
with $n= 50$, $100$, and $200$. 
We record the proportion of times out of 1,000 replications 
when there exist roots with absolute value larger than 1. 
See Table \ref{Table.Cauchy.root}. The percentage
does not get smaller as the sample size increases.

\begin{table}[ht]
\caption{Proportion of times the Cauchy score function admits extraneous roots}
\label{Table.Cauchy.root}
\centering
\medskip
\begin{tabular}{c c}
\hline\hline
$n$ & \% of roots outside $[-1,1]$ \\
\hline
50 & 27.2\\
100 & 26.0\\
200 & 25.7\\
\hline
\hline
\end{tabular}
\end{table}

If we locate all roots of the score equation,
we can then use the parametric likelihood to identify the global maximum. 
The parametric global maximum is a consistent estimator. 
Our interest of this model is: are our proposed two approaches effective? 

Directly defining a profile EL function $\el(\theta)$ with estimating function 
\eqref{Cauchy.score} leads to multiple global maximum of $\el(\theta)$.
Because Cauchy score function does not satisfy conditions C1-C5, 
our global consistency conclusion does not apply.
We suggest that the global consistency remedy work in this case.
Suppose in addition to the knowledge that \eqref{Cauchy.score}
is an unbiased estimating function, we also know that $F$ is symmetric.
We may add $(x-\theta)^{1/3}$ as another component
of the estimating function. This leads to
estimating function \eqref{Cauchy.eq2}.
This choice has also reflected the prior knowledge
that the distribution of $X$ has a long tail.
  
We restrain from verifying that \eqref{Cauchy.eq2} satisfies C1-C5. 
The process is routine and similar to the one for the non-linear regression  model. 
Our simulation obtains the MELE with estimating function \eqref{Cauchy.eq2}
and parametric MLE of $\theta$ based on data from standard Cauchy distribution.
The number of repetitions is $N=1000$. 
We calculate the mean and variance of 1,000 estimates
and present the results in Table \ref{Cauchy.MELE}. 
Because there may be multiple local maxima of both the EL and likelihood function, 
we start with a fixed set of initial values and use the Newton-Raphson algorithm to search for all local maxima. 
Then we compare the values of all local maxima to make sure we correctly locate the MLE and MELE.


\begin{table}[ht]
\caption{Mean and variance of MELE and MLE based on 1,000 simulated data sets}
\centerline{Cauchy example}
\centering
\label{Cauchy.MELE}
\medskip
\begin{tabular}{c c c c c c}
\hline\hline
 & \multicolumn{2}{c}{MELE} & & \multicolumn{2}{c}{MLE} \\\cline{5-6}\cline{2-3}
$n$ & Mean & Variance & & Mean & Variance\\
\hline
50 & -0.0116 & 0.0482 & & -0.0050 & 0.0411\\
100 & -0.0030 & 0.0211 & &-0.0020 & 0.0202 \\
200 & 0.0039 & 0.0101 & & -0.0035 & 0.0098\\
\hline
\hline
\end{tabular}
\end{table}

As expected, the simulation results indicate that both the means of MELE and 
MLE are reasonably close to the true parameter value $\theta^* = 0$ with 
indistinguishable differences in mean and variance.
According to \cite{QinLawless1994}, the MELE is an optimal estimator
of $\theta$ if $g(x; \theta)$ contains the score function.
Adding more unbiased estimating functions
to the score function does not lead to more efficient estimator.
Yet our result reveals an interesting phenomenon:
adding unbiased estimating functions
achieves global consistency and dispels nuisance local maxima.
Figure \ref{Figure.cauchy}  demonstrates that the MELE is
asymptotically normal and well behaved.

\begin{figure}
\label{Figure.cauchy}
    \centering
    \includegraphics[width=0.6\textwidth]{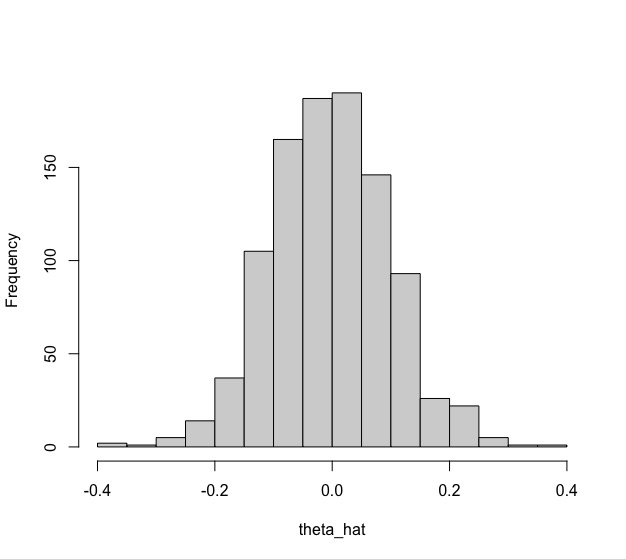} 
    \caption{Histogram of the MELE under Cauchy model, $n=200.$}
    \label{fig:1}
\end{figure}

\vskip 1em
\noindent
{\bf Non-linear regression model \eqref{regr.model}}.
We generate a random sample of size $n$ from $N(5, 4)$ for covariate $x$
and a random sample of size $n$ from $N(0,1)$ for error $\epsilon$. 
We let $\theta=1$ and obtain $y_i =\theta +\theta^2x_i+\epsilon_i$
as specified. These choices are the same as \cite{Veall}.

We define the EL function $el_n(\theta)$
with estimating function given by \eqref{regr.g}.
Given a random generated data set, most often,
$el_n(\theta)$ has a global maximum close to $\theta^* = 1$
and  a local maximum close to $\theta =  -1$. 
In this experiment, we use the proposed Global Maximum Test
to see if a local maximum is the global maximum of $el_n(\theta)$. 
We also use the test proposed in \cite{DeHaan} to see if a local minimum is the global minimum of the least square objective function.
We set the nominal rejection rate to be $\alpha$: the
percentage of global maxima being rejected.
This study generates $N=1,000$ data sets and we record the 
percentage of times when a local or a global maximum is rejected.
We experiment with $n=30, 50$ and $100$. For test of \cite{DeHaan}, 
we uniformly sample $m$ points in the interval $[-5,5]$ with $m=50, 100$. 
The simulation results in Table \ref{Table.regr.var} indicate that the variances of 
MELE and MLE are very close when $n$ is larger than 50.
As seen in Table \ref{Table.regr}, the rejection rates at local maxima of the proposed test are much higher than those of the test of \cite{DeHaan}. We conclude that our proposed test is clearly superior.

\begin{table}[ht]
\caption{Mean and variance of MELE and MLE under non-linear regression model}
\label{Table.regr.var}
\centering
\medskip
\begin{tabular}{c c c c c c}
\hline\hline
 & \multicolumn{2}{c}{MELE} & & \multicolumn{2}{c}{MLE} \\\cline{5-6}\cline{2-3}
$n$ & Mean & Variance & & Mean & Variance\\
\hline
30 & 0.9993 & $2.98\times 10^{-4}$ & & 0.9994 & $2.76\times 10^{-4}$\\
50 & 0.9999 & $1.52\times 10^{-4}$ & & 0.9995 & $1.50\times 10^{-4}$\\
100 & 1.0000 & $7.02\times 10^{-5}$ & & 1.0000 & $6.92\times 10^{-5}$\\ 
\hline
\hline
\end{tabular}

\end{table}
\begin{table}[ht]
\caption{Rejection rates of the proposed test and the test of \cite{DeHaan}}
\centerline{Non-linear regression model example, $\alpha=0.05$}
\label{Table.regr}
\begin{tabular}{c c c c c c c c c}
\hline\hline
 & \multicolumn{2}{c}{Proposed Test} & 
    & \multicolumn{2}{c}{Test of DeHaan ($m=50$)} 
 &  & \multicolumn{2}{c}{Test of DeHaan ($m=100$)}\\\cline{8-9}\cline{5-6}\cline{2-3}
$n$ & Global & Local & & Global & Local & & Global & Local\\
\hline
30 &  0.096 & 0.976 & & 0.032 & 0.482 & & 0.043  & 0.694 \\
50 &  0.075 & 1 & & 0.035 & 0.508 & & 0.032 & 0.748 \\
100 & 0.067 & 1 & & 0.027 & 0.510 & & 0.023 & 0.772 \\
\hline
\hline
\end{tabular}
\end{table}

\vskip 1em
\noindent
{\bf Curved exponential family $N(\theta, \theta^2)$}.
We now compare the performance of the proposed Global Maximum Test
with the test proposed by \cite{Jiang1999}.
 We generate a random sample of size $n$ from $N(1, 1)$.
The population distribution is regarded as a member of curved
parametric family $N(\theta, \theta^2)$ 
when the test of \cite{Jiang1999} is used.
Under nonparametric model, we assume the knowledge of
unbiased estimating function  \eqref{curve.funct}.

Both the nonparametric EL function and the 
parametric likelihood function have two local maxima. 
We generate $N=1000$ data sets and record the 
percentage of times when a local or a global maximum is rejected
for both approaches.

The simulation results in Table \ref{Table.curved} indicate that the 
proposed test rejects much less often of global maximum, 
and far more often of local maximum
than the test of \cite{Jiang1999} across all sample sizes considered. 
We summarize the mean and variance of MELE and MLE in Table \ref{Table.curved.var}. 
The variances of MELE and MLE are almost identical even when $n$ is small.

\begin{table}[ht]
\caption{Mean and variance of MELE and MLE under curved normal model}
\label{Table.curved.var}
\centering
\medskip
\begin{tabular}{c c c c c c}
\hline\hline
 & \multicolumn{2}{c}{MELE} & & \multicolumn{2}{c}{MLE} \\\cline{5-6}\cline{2-3}
$n$ & Mean & Variance & & Mean & Variance\\
\hline
30 & 0.9879 & 0.0108 & & 0.9954 & 0.0104\\
50 & 0.9955 & 0.0069 & & 0.9985 & 0.0068\\
100 & 0.9987 & 0.0032 & & 0.9995 & 0.0032\\  
\hline
\hline
\end{tabular}
\end{table}

\begin{table}[ht]
\caption{Rejection rates of the proposed test and the test of \cite{Jiang1999}}
\centerline{Curved normal model, $\alpha=0.05$}
\label{Table.curved}
\begin{tabular}{c c c c c c}
\hline\hline
 & \multicolumn{2}{c}{Proposed Test} & & 
       \multicolumn{2}{c}{Test of \cite{Jiang1999}} \\\cline{5-6}\cline{2-3}
$n$ & Global & Local & & Global & Local\\
\hline
30 &  0.076 & 1 & & 0.181 & 0.202\\
50 &  0.050 & 1 & & 0.189 & 0.327\\
100 & 0.048 & 1 & & 0.168 & 0.553\\
\hline
\hline
\end{tabular}
\end{table}

\vskip 1em
\noindent

{\bf Mixture Model}
As mentioned in section 4.3, the EL function with three score functions makes it easy to identify the global maximum of the EL function since all local maxima of the EL function tend to be much lower than the global maximum. Or equivalently, all local minima of the profile EL function tend to be much higher than the global minimum.

Figure \ref{fig:2} shows the plot of the profile EL function with a randomly generated sample 
of size 100 from $0.4N(0,1)+0.6N(10,16)$. 
It appears that the global minimum is much lower than other local minima. 
In addition, according to \cite{QinLawless1994}, adding another two score functions 
does not change the asymptotic variance of MELE. 
We support this claim by simulation results with estimating function
\[
g(x;\theta)
=
\begin{pmatrix}
s(x;\theta)\\s(x;\mu_2)\\s(x;p)
\end{pmatrix}.
\]
We compare the performance of  the proposed {\sc Global Maximum Test}
with the test of \cite{Jiang1999}.
We generate a random sample of size $n$ from $0.4N(0,1)+0.6N(10,16)$.
We replicate the simulation $N=1,000$ times. 
We assume all other parameters of the mixture model are known except 
for the centre of the first mixture component. 

\begin{figure}
    \centering
    \includegraphics[width=0.6\textwidth]{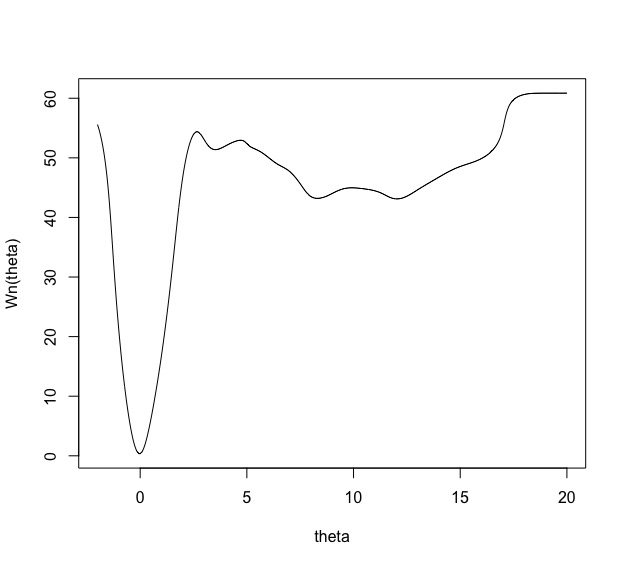} 
    \caption{One typical profile EL function with $g(x;\theta)$ under mixture model, $n=100$}
    \label{fig:2}
\end{figure}

For each data set, we apply two tests to
the global maximum and the largest local maximum located inside the interval $[-2,20]$. 
When $\theta$ is far away from $\theta^*$, the convex hull of $g(x;\theta)$ may not contain 0. 
To avoid this problem, we use the AEL of \cite{Asokan} with a pseudo point $g_{n+1}=-\bar{g}_n$.
This option does not change the first order asymptotic properties of
the approaches we proposed. 
We refer readers interested in AEL to this paper.

We record the rejection rates in Table \ref{Table.mixture}. 
We also summarize the mean and variance of the MELE and MLE in Table \ref{Mixture.MELE}. 
We note that the $W_n(\hat{\theta})$ has a $\chi^2(2)$ limiting distribution with two added score functions. 
The proposed test has well controlled rejection rates
of the global maximum and much higher reject rates
of the local maximum compared to the test of \cite{Jiang1999}. 
According to Table \ref{Mixture.MELE}, the bias and variance of MELE and MLE are almost identical across all sample sizes considered. 

Alternatively, we may also only combine the score functions for two mixture means. The resulting $W_n(\hat{\theta})$ has a $\chi^2(1)$ limiting distribution. 
However, with three score functions, we have maximal information about the mixture model. 
It turns out that it is easier to differentiate the global minimum from the local minima. 
The MELE also has smaller bias and variance in general.

\begin{table}[ht]
\caption{Rejection rates of the proposed test and the test of  \cite{Jiang1999}}
\centerline{Mixture model, $\alpha=0.05$}
\label{Table.mixture}
\begin{tabular}{c c c c c c}
\hline\hline
 & \multicolumn{2}{c}{Proposed Test} & & 
       \multicolumn{2}{c}{Test of \cite{Jiang1999}} \\\cline{5-6}\cline{2-3}
$n$ & Global & Local & & Global & Local\\
\hline
100 &  0.046 & 1 & & 0.087 & 0.194\\
200 &  0.049 & 1 & & 0.061 & 0.668\\
500 &  0.064 & 1 & & 0.059 & 0.988\\
\hline
\hline
\end{tabular}
\end{table}

\begin{table}[ht]
\caption{Mean and variance of MELE and MLE based on 1,000 simulated data sets}
\centerline{Mixture model}
\centering
\label{Mixture.MELE}
\medskip
\begin{tabular}{c c c c c c}
\hline\hline
 & \multicolumn{2}{c}{MELE} & & \multicolumn{2}{c}{MLE} \\\cline{5-6}\cline{2-3}
$n$ & Mean & Variance & & Mean & Variance\\
\hline
100 & 0.0132 & 0.0288 & & 0.0132 & 0.0288\\
200 & 0.0069 & 0.0136 & & 0.0068 & 0.0136 \\
500 & 0.0041 & 0.0055 & & 0.0041 & 0.0055\\

\hline
\hline
\end{tabular}
\end{table}
\vskip 1em

\section{Conclusion and discussions}
This paper identifies the conditions under which the global maximum of the 
empirical likelihood function is strongly consistent. 
When these conditions are not satisfied under a common setting, 
we propose a generic global consistency remedy. 
The remedy leads to new profile EL that satisfies these conditions
and has a unique consistent global maximum. 
After the adoption of the global consistency remedy, 
the profile EL function can still have multiple local maxima.
We propose to avoid local maximum with a novel global maximum test.
Our proposed global maximum test  has well-controlled rejection rates 
for the global maximum and much higher power to detect the non-global local maximum 
compared to the tests of \cite{Jiang1999} and \cite{DeHaan}.
Our global consistency remedy often identifies
the model under investigation as a sub-model of a general model.
We supplement the existing set of estimating functions in view of the
general model. This is a new technique that is broadly applicable.

\vskip 3em

\begin{appendix}

\noindent
{\bf Appendix.} This section gives detailed proofs of Lemma
\ref{lemma2.3} and Theorem \ref{global.consistency}.

\vspace{1em}
\noindent
{\bf Proof of Lemma \ref{lemma2.3} (a)}.
By the Lipschitz continuity condition (C3),
for any $\theta' \in B(\theta, \rho)$, we have
\[
\norm{\bar g_n(\theta') - \bar g_n(\theta)} 
\leq 
n^{-1} \sum_{i=1}^{n} \norm{g(x_i; \theta') -  g(x_i; \theta)}
\leq 
n^{-1} \norm{\theta'-\theta} \sum_{i=1}^{n}G(x_i).
\]
Hence, by the strong law of large numbers, we have
\[
\sup_{\theta'\in B(\theta,\rho)}\norm{\bar g_n(\theta')-\bar g_n(\theta)}
\leq 
\frac{\rho}{n}\sum_{i=1}^{n}G(x_i)=\rho E[G(X)]+o(1).
\]
Consequently,
\[
\sup_{\theta' \in B(\theta, \rho)}
\norm{\widehat{\lambda}(\theta')-\widehat{\lambda}(\theta)} 
= 
n^{-3/4}\sup_{\theta'\in B(\theta,\rho)}\norm{\bar g_n(\theta')-\bar g_n(\theta)}
=
O(n^{-3/4}).
\]
Since $\widehat{\lambda}(\theta)=n^{-3/4}\bar g_n(\theta)=O(n^{-3/4})$, 
we conclude
\[
\sup_{\theta'\in B(\theta,\rho)}\norm{\widehat{\lambda}(\theta')}=O(n^{-3/4}).
\]
This completes the proof of Lemma \ref{lemma2.3} (a).

\vspace{1em}
\noindent
{\bf Proof of Lemma \ref{lemma2.3} (b)}.
Similarly, since $\norm{g(x_i;\theta') - g(x_i; \theta)}\leq \rho G(x_i)$,
we have
\[
\sup_{\theta'\in B(\theta, \rho)} \norm{g(x_i;\theta')} 
\leq 
\norm{g(x_i;\theta)}+\rho G(x_i).
\]
Applying Lemma \ref{lemma.owen} to order statistics of
$g(x_i;\theta)$ and $G(x_i)$,  we get
\[
\max_i \sup_{\theta'\in B(\theta,\rho)} \norm{g(x_i;\theta')}
\leq 
\max_i \{ \norm{g(x_i;\theta)}+\rho G(x_i) \}
= o(n^{1/3}).
\]
Therefore, 
\[
\max_{1\leq i \leq n} 
\sup_{\theta' \in B(\theta,\rho)} g(x_i;\theta')
=
o(n^{1/3}).
\]
This completes the proof of Lemma \ref{lemma2.3} (b).

\vspace{1em}
\noindent
{\bf Proof of Lemma \ref{lemma2.3} (c)}.
Based on moment conditions C2 and C3 and the strong law of
large numbers, we have both 
\[
\sum_{i=1}^n g(x_i; \theta) = O(n);  ~~\sum_{i=1}^n G(x_i) = O(n).
\]
Using triangular inequality, we have
\[
\sum_{i=1}^n \norm{g(x_i;\theta')}^2
\leq 
\sum_{i=1}^n \norm{g(x_i;\theta)}^2
+
\rho^2 \sum_{i=1}^n \norm{G(x_i)}^2
=  O(n).
\]
Note that the upper bound does not depend on $\theta'$
other than requiring $\theta' \in B(\theta, \rho)$.
This completes the proof of Lemma \ref{lemma2.3} (b).

\vspace{1em}
\noindent
{\bf Proof of Lemma \ref{lemma2.3} (d)}.

Suppose the convex condition holds and
$\lambda^*(\theta) \in \mathbb{D}(\theta)$ 
is a value depends on $\theta$ and solves \eqref{Lag.eqn}. 
In this case, $\Omega_n(\lambda^*(\theta); \theta) = W_n(\theta)$.

Applying Taylor's expansion,  for all $\theta' \in B(\theta,\rho)$, we have
\begin{equation}
\label{Taylor.ser}
\sum_{i=1}^{n}\log[1+\widehat{\lambda}^{\tau}(\theta') g(x_i; \theta')]
= 
n \widehat{\lambda}^{\tau}(\theta') \bar g_n(\theta') - \frac{1}{2} \epsilon_n(\theta'),
\end{equation}
with the remainder term given by
\[
\epsilon_n(\theta')
=
\widehat{\lambda}^{\tau}(\theta')
\sum_{i=1}^{n}\frac{g(x_i; \theta')  g^{\tau}(x_i;\theta')}{(1+\xi^\tau (\theta') g(x_i;\theta'))^2}
\widehat{\lambda}(\theta'),
\]
and $\xi(\theta')$ being a value between 0 and $\widehat{\lambda}(\theta')$. 
We next assess the order of two terms.

With $\rho$ small enough and $\bbE[g(X; \theta)] \neq 0$,
we have $\norm{\bbE[g(X; \theta)]} - \rho \bbE[G(X)]= \delta < 0$.
Activating Lipschitz condition and the strong law of large numbers, we find
\[
\inf_{\theta'\in B(\theta,\rho)} \norm{\bar g_n(\theta')} \geq \delta+o(1).
\]
Hence, almost surely, we have
\begin{equation}
\label{first.term}
\inf_{\theta'\in B(\theta,\rho)}
\big[ \widehat{\lambda}^{\tau}(\theta') \bar g_n (\theta') \big ]
=
n^{-3/4}\inf_{\theta'\in B(\theta,\rho)} \norm{\bar g_n(\theta')}^2
\geq 
n^{-3/4}\delta^2.
\end{equation}

To assess the order of $\epsilon_n(\theta')$, we first note
\[
\sup_{\theta'\in B(\theta,\rho)}\norm{\xi(\theta')}
\leq 
\sup_{\theta'\in B(\theta,\rho)}\norm{\widehat{\lambda}(\theta)}
=O(n^{-3/4}).
\]
Further, recall Lemma \ref{lemma2.3} (b) that
\[
\max_{1\leq i\leq n}\sup_{\theta'\in B(\theta,\rho)} \norm{g(x_i;\theta')}=o(n^{1/3}).
\]
Hence, we have
\[
\sup_{\theta'\in B(\theta,\rho)} \max_{1\leq i\leq n}\norm{\xi(\theta') g(x_i;\theta')} 
=o(1).
\]
Consequently, almost surely, we have
\begin{equation}
\label{second.term}
\sup_{\theta'\in B(\theta,\rho)} \norm{\epsilon_n(\theta')} 
\leq
\{ 1+o(1)\} \sup_{\theta'\in B(\theta,\rho)}
\norm{\widehat{\lambda}(\theta')}^2 \sum_{i=1}^{n}\norm{g(x_i;\theta')}^2
=
n^{-1/2} \delta^2.
\end{equation}
Combining two order assessments \eqref{first.term} and \eqref{second.term}, we
arrive at the conclusion
\[
W_n(\theta; \rho)
\geq
\inf_{\theta'\in B(\theta,\rho)}
[  n \widehat{\lambda} \bar g_n(\theta')  -  \norm{\epsilon_n(\theta')} ]
\geq 
n^{1/4}[\delta + o(1)].
\]
Hence $W_n(\theta,\rho) \to \infty$ almost surely. 
This completes the proof of Lemma \ref{lemma2.3} (d).


\vspace{1em}
\noindent
{\bf Proof of Lemma \ref{lemma2.3} (e)}.
Without loss generality, let $b(\theta)=1$.

We first show that there exists $r > 0$ such that
\begin{equation}
\label{eqn.lemma2.3.e}
\max_{1\leq i\leq n}
\sup_{\norm{\theta}\geq r}\norm{\widehat {\lambda}^{\tau}g(x_i;\theta)}=o(1).
\end{equation}
This conclusion would imply
\[
\min_{1\leq i\leq n}\inf_{\norm{\theta}\geq r}
  	[1+\widehat {\lambda}^{\tau} g(x_i;\theta)] > 0
\]
almost surely which ensures $W_n(\widehat{\lambda}, \theta)$ 
is well-defined for all $\norm{\theta} \geq r$.

To prove \eqref{eqn.lemma2.3.e}, we first note that
\[
\norm{\widehat {\lambda}^{\tau}(\theta) g(x_i;\theta)}
\leq
\norm{\widehat {\lambda}(\theta)} \norm{g(x_i; \theta)}.
\]
We now assess the order of two factors on the right-hand side of
the above inequality over ${\norm{\theta}\geq r}$ and $1 \leq i \leq n$.
By (ii) in Condition C5, there exists an $\epsilon > 0$ so that for
all sufficiently large $r$,
\[
\inf_{\norm{\theta}\geq r}\norm{\bbE[g(X; \theta)]} \geq 2\epsilon.
\]
Let $\theta_0 \in \Theta$ be a parameter value such as $\norm{\theta_0}=r$. 
By the law of large numbers, we have
\[
\bar g_n(\theta_0) = n^{-1} \sum_{i=1}^{n} g(x_i; \theta_0 )  \to \bbE [g(X; \theta_0)]
\]
almost surely, which implies
\(
\norm{\bar g_n(\theta_0)} \geq 2\epsilon.
\)
By (iii) of (C5), we can choose $r$ large enough such that
$ \bbE[\Delta(X; r)]<\epsilon $.
Hence, for all $\theta$ such as $\norm{\theta}\geq r$, 
\[
\norm{\bar g_n(\theta)}
\geq 
\norm{\bar g_n(\theta_0 )}- n^{-1} \sum_{i=1}^{n}\Delta(x_i; r)
\geq 
2\epsilon-\epsilon
= \epsilon
\]
almost surely. Apparently, this inequality holds uniformly over 
$\norm{\theta}\geq r$.

At the same time,  by C5 (i), we have
\[
\sup_{\norm{\theta}\geq r}\norm{\bar g_n(\theta)}
\leq
n^{-1} \sum_{i=1}^{n}\sup_{\norm{\theta}\geq r}\norm{g(x_i;\theta)}
=O(1).
\]
Applying Lemma 11.2 in \cite{Owen2001} under C5 (i), 
\[
\max_{1\leq i\leq n}\sup_{\norm{\theta}\geq r}\norm{g(x_i;\theta)}=o(n^{1/3}).
\]
Recall $\widehat \lambda(\theta) = n^{-3/4} \bar g_n(\theta)$.
Hence, uniformly over $1 \leq i \leq n$, we have
\[
\sup_{\norm{\theta}\geq r}\norm{\widehat{\lambda}^{\tau}g(x_i;\theta)}
=
O(n^{-3/4}) \times o(n^{1/3}) =o(1).
\]
Hence $\overline{W}_n(r)$ is well defined in probability. 

Recall that when  the convex hull condition holds, we have
\[
W_n(\theta) \geq \Omega_n(\widehat{\lambda}, \theta)
=
\sum_{i=1}^{n}\log[1+\widehat{\lambda}^{\tau}g(x_i;\theta)].
\]
Applying first order Taylor expansion as we did in \eqref{Taylor.ser}, we get
\[
\Omega_n(\widehat{\lambda}, \theta)
= 
\sum_{i=1}^{n}\widehat{\lambda}^{\tau}g(x_i;\theta) - (1/2) \epsilon_n(\theta),
\]
with
\[
\epsilon_n(\theta)
= \widehat{\lambda}^{\tau}(\theta)
\sum_{i=1}^{n}\frac{g(x_i;\theta)g^{\tau}(x_i;\theta)}{(1+\xi(\theta) g(x_i;\theta))^2}
\widehat{\lambda}(\theta),
\]
and $\xi(\theta)$ being a vector between {\bf 0} and $\widehat{\lambda}(\theta)$.
 Hence
\[
\overline{W}_n(r)
\geq 
\inf_{\norm{\theta}\geq r}\Omega_n(\widehat{\lambda},\theta)
\geq 
\inf_{\norm{\theta}\geq r}\Big[\sum_{i=1}^{n}\widehat{\lambda}^{\tau}g(x_i;\theta)\Big]
- 
\sup_{\norm{\theta}\geq r}\epsilon_n(\theta).
\]
We have
\[
\inf_{\norm{\theta}\geq r}\Big[ \sum_{i=1}^{n}\widehat{\lambda}g(x_i;\theta)\Big]
\geq
n^{1/4}\inf_{\norm{\theta}\geq r}\norm{\bar g_n(\theta)}^2
\geq 
n^{1/4}\{\epsilon + o(1)\}.
\]
Since $\norm{\xi(\theta)} \leq \norm{\widehat{\lambda}(\theta)}$,
we have 
\[
\sup_{\norm{\theta}\geq r}\max_{1\leq i\leq n}\norm{\xi(\theta) g(x_i;\theta)}
\leq 
\sup_{\norm{\theta}\geq r}\norm{\xi(\theta)}
 \max_{1\leq i\leq n}\sup_{\norm{\theta}\geq r}\norm{g(x_i;\theta)}=o(1).
\]
With this, we find
\begin{align*}
\sup_{\norm{\theta}\geq r}\epsilon_n(\theta) 
&\leq  
\{1 + o(1)\}\sup_{\norm{\theta}\geq r}[\widehat{\lambda}^{\tau} 
	\sum_{i=1}^{n}g(x_i;\theta) g^{\tau}(x_i; \theta)\widehat{\lambda}] \\
&\leq  
\{1 + o(1)\} n^{-3/2} \sup_{\norm{\theta}\geq r} \norm{\bar g_n(\theta)}^2 \times
	\sup_{\norm{\theta}\geq r} \sum_{i=1}^{n}\norm{g(x_i;\theta)}^2 \\
& =
O(n^{-3/2})\sup_{\norm{\theta}\geq r}\sum_{i=1}^{n}\norm{g(x_i;\theta)}^2\\
& = 
O(n^{-1/2})
\end{align*}
by condition C5 (i).
Therefore, $\overline{W}_n(r) \to \infty$ at a rate of at least $n^{1/4}$ almost surely.

When the convex hull condition is violated, 
$W_n(\theta)=-\infty$ so the conclusion holds automatically. 
This completes the proof.
\end{appendix}

\bibliographystyle{imsart-nameyear} 
\bibliography{global_consistency}       
\nocite{*}

\end{document}